\documentclass[a4paper,11pt]{amsart}
\usepackage{amssymb}
\usepackage{mathptmx}

\title[Towards boundedness of minimal log discrepancies]
{Towards boundedness of minimal log discrepancies \\ by Riemann--Roch theorem}
\author{Masayuki Kawakita}

\address{Research Institute for Mathematical Sciences, Kyoto University, Kyoto 606-8502, Japan}
\email{masayuki@kurims.kyoto-u.ac.jp}

\newtheorem{theorem}{Theorem}[section]
\newtheorem{proposition}[theorem]{Proposition}
\newtheorem{lemma}[theorem]{Lemma}
\newtheorem{conjecture}[theorem]{Conjecture}
\newtheorem{problem}[theorem]{Problem}

\newcommand{\bL}{{\mathbb L}}
\newcommand{\bQ}{{\mathbb Q}}
\newcommand{\bR}{{\mathbb R}}
\newcommand{\cJ}{{\mathcal J}}
\newcommand{\cO}{{\mathcal O}}
\newcommand{\fm}{{\mathfrak m}}

\newcommand{\mld}{\operatorname{mld}}
\newcommand{\mult}{\operatorname{mult}}
\newcommand{\ord}{\operatorname{ord}}
\newcommand{\Spec}{\operatorname{Spec}}

\begin{document}
\begin{abstract}
We introduce an approach of Riemann--Roch theorem to the boundedness problem of minimal log discrepancies in fixed dimension. After reducing it to the case of a Gorenstein terminal singularity, firstly we prove that its minimal log discrepancy is bounded if either multiplicity or embedding dimension is bounded. Secondly we recover the characterisation of a Gorenstein terminal three-fold singularity by Reid, and the precise boundary of its minimal log discrepancy by Markushevich, without explicit classification. Finally we provide the precise boundary for a special four-fold singularity, whose general hyperplane section has a terminal piece.
\end{abstract}

\maketitle

A primary study in birational geometry is to find and analyse a good representative in each birational equivalence class, and the minimal model program has been formulated to find the representative by comparison of canonical divisors. It works at present with provision of termination of flips by the work \cite{HM05} of Hacon and McKernan, or in the case of big boundaries by their following work \cite{BCHM06} together with Birkar and Cascini, whereas the termination in the relatively projective case is reduced by Shokurov in \cite{S04} to two conjectures on minimal log discrepancy, a numerical invariant attached to a singularity. Although the two conjectures are believed very difficult, each of them leads as its corollary another related but more accessible conjecture, the boundedness (BDD), that there exists an upper bound of all minimal log discrepancies in fixed dimension. In fact Shokurov has conjectured its precise boundary in \cite{S88}. The purpose of this paper is to introduce an approach of Riemann--Roch theorem to this (BDD).

It is standard to reduce (BDD) to the case of a Gorenstein terminal singularity, Proposition \ref{prp:reduction}. Then we derive (BDD), Theorem \ref{thm:emb_mult_mld}, on the assumption of bounded multiplicity or embedding dimension by the basic property of Riemann--Roch formula that a kind of multiplicity appears in its top term. This theorem would be attractive from the perspective that the minimal log discrepancy measures how singular a variety is. It should be large only in the case of mild singularities, like those with small multiplicity or embedding dimension. This is supported also by the description \cite{EMY03} of minimal log discrepancies in terms of motivic integration by Ein, Musta\c{t}\v{a} and Yasuda.

We proceed the approach of Riemann--Roch theorem by focusing on the second-top term as well as the top one, in which an intersection number with the canonical divisor appears. It is essentially an argument on a surface obtained by cutting out with general hyperplane sections. Besides formulae related to log discrepancies, we derive an interesting property, Proposition \ref{prp:artinian}, of an artinian ring obtained in the same manner, that the power of its maximal ideal to the dimension of the original singularity vanishes.

Now this approach generates an extremely simple proof, Theorem \ref{thm:dim3}, of the results on a Gorenstein terminal three-fold singularity due to Reid and Markushevich in reverse order without explicit classification. It is immediate to deduce the precise boundary, three, \cite{Mk96} of its minimal log discrepancy by Markushevich, without its characterisation. Then it recovers the characterisation \cite{R83} by Reid that its general hyperplane section is canonical, without the detailed study \cite{L77}, \cite{R75} of elliptic surface singularities.

Sadly, we bound minimal log discrepancies in dimension four, Theorem \ref{thm:dim4}, just in a special case, when the singularity has a terminal-like hyperplane section. This limitation is presumably due to lack of the study of special hyperplane sections of a singularity. We provide a few remarks towards an advance of our approach.

A part of this research was achieved during my visit at University of Cambridge. I should like to thank Dr C.\ Birkar for his warm hospitality. Partial support was provided by Grant-in-Aid for Young Scientists (A) 20684002.

\section{Conjectures on minimal log discrepancies}
We work over an algebraically closed field $k$ of characteristic zero throughout. A \textit{pair} $(X,\Delta)$ consists of a normal variety $X$ and a \textit{boundary} $\Delta$ on $X$, which is an effective $\bR$-divisor such that $K_X+\Delta$ is an $\bR$-Cartier $\bR$-divisor. A valuation of the function field of $X$ is called an \textit{algebraic valuation} if it is defined by a prime divisor $E$ on a resolution $\bar{X}$ of singularities of $X$. It is denoted by $v_E$, and the image of $E$ in $X$ is called the \textit{centre} of $v_E$ on $X$. Writing the pull-back of $K_X+\Delta$ to $\bar{X}$ as the sum of $K_{\bar{X}}$ and an $\bR$-divisor $\bar{\Delta}$ whose push-forward to $X$ is $\Delta$, we define the \textit{log discrepancy} $a_E(X,\Delta)$ of $v_E$ with respect to $(X,\Delta)$ as one minus the coefficient of $E$ in $\bar{\Delta}$. For a closed subset $Z$ of $X$, the \textit{minimal log discrepancy} $\mld_Z(X,\Delta)$ of $(X,\Delta)$ over $Z$ is the infimum of $a_E(X,\Delta)$ for all algebraic valuations whose centres are in $Z$. It is either a non-negative real number or minus infinity in dimension at least two, but for convenience we set $\mld_Z(X,\Delta):=-\infty$ even in dimension one if the infimum is negative. The $\mld_Z(X,\Delta)$ is an invariant on the formal scheme of $X$ along $Z$ as remarked in \cite[Theorem 3.2]{K08}. We write as $a_E(X)$ and $\mld_ZX$ simply when the boundary is zero. One should refer to \cite{KM98} for the definitions of ((kawamata, purely, divisorially) log) terminal and (log) canonical singularities, formulated in terms of log discrepancies.

For a pair $(X,\Delta)$, every closed subset $Z$ of $X$ is stratified into a finite union of irreducible constructible subsets $Z_i$ such that $\mld_Z(X,\Delta)$ is equal to the minimum of $\mld_{x_i}(X,\Delta)-\dim Z_i$, where $x_i$ is a general closed point of $Z_i$. Thus we are principally interested in the case when $Z$ is a closed point, and henceforth we consider minimal log discrepancies over closed points only. One can generalise Conjectures \ref{cnj:LSCandACC}, \ref{cnj:PIA} and Problem \ref{prb:BDD} below straightforwardly.

Our main motivation to study minimal log discrepancies is their role in the minimal model program. Birkar, Cascini, Hacon and McKernan in \cite{BCHM06}, combining a former work \cite{HM05} or \cite{HM08} of the last two authors, proved that this program runs at least in one direction when the boundary is relatively big, from which the existence of flips follows. Now the minimal model program works with provision of termination of flips, and the termination in the relatively projective case has been reduced to the two conjectures below on minimal log discrepancies by Shokurov in \cite{S04}.

\begin{conjecture}\label{cnj:LSCandACC}
\setlength{\leftmargini}{2em}
\begin{enumerate}
\item
\textup{(LSC, lower semi-continuity \cite[Conjecture 2.4]{Am99})} \
For a pair $(X,\Delta)$, the function on the set of closed points of $X$ sending $x$ to $\mld_x (X,\Delta)$ is lower semi-continuous.
\item\label{itm:ACC}
\textup{(ACC, ascending chain condition \cite{S88}, \cite[Conjecture 4.2]{S96})} \
Fix $n$ and a finite sequence $\{d_i\}$ of real numbers. Then the set of all $\mld_x(X,\Delta)$ for pairs $(X,\Delta)$ of dimension $n$ such that $\Delta$ has its irreducible decomposition $\sum_i d_i\Delta_i$, satisfies the ascending chain condition.
\end{enumerate}
\end{conjecture}

Although the statement (\ref{itm:ACC}), or a more general one, is a standard formulation of (ACC), a weaker version of it besides (LSC) is enough to derive the termination of flips, where a pair $(X,\Delta)$ is fixed and we consider the set of all $\mld_{x'}(X',\Delta')$ for $(X',\Delta')$ obtained from $(X,\Delta)$ by a sequence of $K_X+\Delta$-flips. We have (LSC) in dimension two by the classification of surface singularities, and (ACC) in dimension two thanks to the deep numerical analysis of surface singularities by Alexeev in \cite{Al93}. In dimension at least three, we have only (LSC) in the case of local complete intersection, which was proved as well as (PIA) stated below by Ein, Musta\c{t}\v{a} and Yasuda in \cite{EM04}, \cite{EMY03} by the theory of motivic integration. 

The first conjecture (LSC) implies that $\mld_x(X,\Delta)$ is bounded from above by the dimension of $X$ as it holds for a smooth point $x$ trivially, whereas the second conjecture (ACC), with a reduction to the case of no boundaries, implies the boundedness of minimal log discrepancies in fixed dimension. Therefore a basic problem towards Conjecture \ref{cnj:LSCandACC} would be the following.

\begin{problem}[BDD, boundedness]\label{prb:BDD}
For each $d$, find a real number $a(d)$ such that all minimal log discrepancies in dimension $d$ are at most $a(d)$.
\end{problem}

In fact Shokurov has conjectured in \cite{S88} the precise boundary $a(d)=d$, and that $\mld_x(X,\Delta)$ should attain $d$ if and only if $x$ is a smooth point outside the support of $\Delta$. The conjecture (BDD), which is discussed in this paper, is not known even in dimension four, and we have had $a(3)=3$ after the explicit classification \cite{R83} of Gorenstein terminal three-fold singularities with \cite{EMY03} or \cite{Mk96}.

As it has been indicated already, there exists one more conjecture on minimal log discrepancies, closely related to those above.

\begin{conjecture}[PIA, precise inversion of adjunction {\cite[Chapter 17]{K+92}}]\label{cnj:PIA}
Let $(X,S+B)$ be a pair such that $S$ is a normal prime divisor not contained in the support of $B$, and $x$ a closed point of $S$. One can construct the different $B_S$ which is a boundary on $S$ such that $K_S+B_S$ is the restriction of $K_X+S+B$ to $S$. Then $\mld_x(X,S+B)=\mld_x(S,B_S)$.
\end{conjecture}

Inversion of adjunction has its origin in the connectedness lemma \cite[17.4 Theorem]{K+92}, \cite[5.7]{S92}, which implies that $(X,S+B)$ is purely log terminal about $S$ if and only if $(S,B_S)$ is kawamata log terminal. It was extended to the equivalence of log canonicity in \cite{K07}. Its precise version (PIA) holds in the case when both $X$ and $S$ are local complete intersection by \cite{EM04}, \cite{EMY03}. Referring to \cite[Chapter 17]{K+92} we may consider a strong version of (PIA), on a variant $\mld'_x(X,S+B)$ of minimal log discrepancy defined as the infimum of $a_E(X,S+B)$ for $v_E$ with $E$ exceptional whose centre intersects $S$ exactly at $x$. Then we expect that $\mld'_x(X,S+B)$ should be equal to $\mld_x(X,S+B)$, hence to $\mld_x(S,B_S)$, and actually it was proved in the purely log terminal case with $\mld'_x(X,S+B)\le1$ by \cite{BCHM06} after the idea in \cite[Chapter 17]{K+92}.

\section{Arbitrary dimension}
Returning to Problem \ref{prb:BDD}, we start with the following standard reduction. Note that this reduction is applicable also to the precise version of (BDD).

\begin{proposition}\label{prp:reduction}
Let $(x \in X,\Delta)$ be a germ of a pair. Then there exists a germ $y\in Y$ of a $\bQ$-factorial Gorenstein terminal singularity of the same dimension as $X$ has such that $\mld_x(X,\Delta) \le \mld_yY$.
\end{proposition}

\begin{proof}
We may assume that $(X,\Delta)$ is log canonical, equivalently $\mld_x(X,\Delta)\ge0$. First we shall construct a projective birational morphism $g\colon W\to X$ from a log terminal variety $W$ and a boundary $\Delta_W$ such that $K_W+\Delta_W=g^*(K_X+\Delta)$. Take a log resolution $f\colon\bar{X}\to X$ of $(X,(\Delta+\Delta')/2)$ for a general effective $\bR$-divisor $\Delta'\sim_{\bR}\Delta$. Let $\bar{\Delta},\bar{\Delta}'$ denote the strict transforms of $\Delta,\Delta'$. According to \cite{BCHM06} we run the minimal model program over $X$ for the kawamata log terminal pair $(\bar{X},(1-\epsilon)\bar{\Delta}+\epsilon\bar{\Delta}')$ for a small positive real number $\epsilon$ to obtain its relative minimal model $g\colon W\to X$. Then every $g$-exceptional divisor has log discrepancy at most one with respect to $(X,(1-\epsilon)\Delta+\epsilon\Delta')$ by the negativity lemma \cite[2.19 Lemma]{K+92}, and so with respect to $(X,\Delta)$ as $\epsilon$ is small, whence we have the desired boundary $\Delta_W$ on $W$.

Consider a germ of $W$ at a closed point $w\in g^{-1}(x)$. Take its index-one cover $\hat{w}\in\hat{W}\to w\in W$, the covering associated to the $\bQ$-Cartier divisor $K_W$, and a $\bQ$-factorial terminalisation $h\colon Y\to\hat{W}$ by \cite{BCHM06}. Then for a closed point $y\in h^{-1}(\hat{w})$ we have $\mld_x(X,\Delta)\le\mld_w(W,\Delta_W)\le\mld_wW\le\mld_{\hat{w}}\hat{W}\le\mld_yY$.
\end{proof}

We shall not use the $\bQ$-factorial property in this paper. We try to bound minimal log discrepancies of Gorenstein terminal singularities, but in this section we allow Gorenstein canonical singularities since this relaxation does not affect any statements.

We have an experimental knowledge that the minimal log discrepancy measures how singular a variety is. For example, a surface singularity is smooth if its minimal log discrepancy is greater than one, is a Du Val singularity if it is at least one, and is a quotient singularity if it is greater than zero. It brings us expecting that Problem \ref{prb:BDD} should be reduced to the case of mild singularities, like those with small multiplicity or embedding dimension. This expectation is supported also by the theory of motivic integration. Roughly speaking, for a scheme $X$ its jet scheme $J_nX$ is the collection of morphisms $\Spec k[t]/(t^{n+1})\to X$, and the arc space $J_\infty X$, the inverse limit of them, is that of morphisms $\Spec k[[t]]\to X$. Set $\pi_n\colon J_nX\to X$, $\pi_{nm}\colon J_mX\to J_nX$. For a Gorenstein canonical singularity $x\in X$ of dimension $d$, the ideal sheaf $\cJ_X$ is the image of the natural map $\Omega_X^d\otimes\cO_X(-K_X)\to\cO_X$. Then the minimal log discrepancy is described as $\mld_x X =-\dim\int_{\pi_\infty^{-1}(x)}\bL^{\ord_{\cJ_X}}d\mu_X$ in terms of motivic integration by \cite{EMY03}. It means that for $j$, $n\gg j$ and a constructible subset $U$ of $\pi_n^{-1}(x)$ on which $\cJ_X$ has constant order $j$, $\mld_xX$ is at most $(n+1)d-j-\dim\pi_{nm}(\pi_{nm}^{-1}(U))$ for $m\gg n$. Hence $\mld_xX$ should be small when $X$ has large jet schemes, particularly large $J_1X$, the total tangent space.

The following theorem supplies (BDD) with provision of the boundedness of multiplicity or embedding dimension.

\begin{theorem}\label{thm:emb_mult_mld}
\setlength{\leftmargini}{2em}
\begin{enumerate}
\item\label{itm:emb_mult}
For each $e$ there exists a number $m(e)$ such that an arbitrary Gorenstein canonical singularity of embedding dimension at most $e$ has multiplicity at most $m(e)$.
\item\label{itm:mult_mld}
A Gorenstein canonical singularity has minimal log discrepancy at most its dimension times its multiplicity.
\end{enumerate}
\end{theorem}

\begin{proof}
Let $x \in X$ be a Gorenstein canonical singularity of dimension $d$. Take a log resolution $f\colon\bar{X}\to X$ such that the strict transform $\bar{H}$ of a general hyperplane section $H$ of $X$ through $x$ is $f$-free. We write $f^*H=\bar{H}+E$, then $\fm_x\cO_{\bar{X}}=\cO_{\bar{X}}(-E)$ for the maximal ideal sheaf $\fm_x$, and $\mult_xX=(E\cdot\bar{H}^{d-1})$ by the Cohen--Macaulay property of $X$. Set $K_{\bar{X}/X}:=K_{\bar{X}}-f^*K_X$, which is an exceptional divisor.

(\ref{itm:emb_mult})
We suppose $d\le\dim\fm_x/\fm_x^2\le e$. Take the exact sequences
\begin{align*}
0 \to \cO_{\bar{X}}(K_{\bar{X}/X}-(l+1)E) \to \cO_{\bar{X}}(K_{\bar{X}/X}-lE) \to \cO_E(K_E-(l+1)E|_E) \to 0
\end{align*}
and consider the polynomial $P(l):=\chi(\cO_E(K_E-(l+1)E|_E))$ of degree $d-1$ in $l$. Since $R^if_*\cO_{\bar{X}}(K_{\bar{X}/X}-lE)=0$ for $i\ge1,l\ge0$ by Kawamata--Viehweg vanishing theorem \cite[Theorem 1-2-3]{KMM87}, we have for $l\ge0$
\begin{align*}
P(l)=\dim f_*\cO_{\bar{X}}(K_{\bar{X}/X}-lE)/f_*\cO_{\bar{X}}(K_{\bar{X}/X}-(l+1)E).
\end{align*}
Because of the canonicity of $X$, the direct image sheaf $f_*\cO_{\bar{X}}(K_{\bar{X}/X}-lE)$ contains $f_*\cO_{\bar{X}}(-lE)$, which contains $\fm_x^l$. Hence the sheaf $f_*\cO_{\bar{X}}(K_{\bar{X}/X}-lE)/f_*\cO_{\bar{X}}(K_{\bar{X}/X}-(l+1)E)$ is a sub-quotient sheaf of $\cO_X/\fm_x^{l+1}$, whose dimension is bounded by $\binom{e+l}{e}$. Therefore the polynomial $P(l)$ has only finite possibilities, whence so does the coefficient $\mult_xX/(d-1)!$ of $l^{d-1}$ in $P(l)$. 

(\ref{itm:mult_mld})
We write $E=\sum_im_iE_i$ for its irreducible decomposition, and choose an $E_0$ such that $\bar{H}|_{E_0}$ is big, equivalently $(E_0\cdot\bar{H}^{d-1})>0$. Take the exact sequences
\begin{align*}
0 \to \cO_{\bar{X}}(K_{\bar{X}/X}-lE) \to \cO_{\bar{X}}(K_{\bar{X}/X}+E_0-lE) \to \cO_{E_0}(K_{E_0}-lE|_{E_0}) \to 0
\end{align*}
and consider the polynomial $Q(l):=\chi(\cO_{E_0}(K_{E_0}-lE|_{E_0}))$ of degree $d-1$ in $l$. A similar application of the vanishing theorem implies for $l\ge1$
\begin{align*}
Q(l)=\dim f_*\cO_{\bar{X}}(K_{\bar{X}/X}+E_0-lE)/f_*\cO_{\bar{X}}(K_{\bar{X}/X}-lE).
\end{align*}
The direct image sheaves $f_*\cO_{\bar{X}}(K_{\bar{X}/X}+E_0-lE)$ and $f_*\cO_{\bar{X}}(K_{\bar{X}/X}-lE)$ are subsheaves of the sheaf $\cO_X$ of regular functions on $X$. Since the only difference between them is $E_0$, the quotient $f_*\cO_{\bar{X}}(K_{\bar{X}/X}+E_0-lE)/f_*\cO_{\bar{X}}(K_{\bar{X}/X}-lE)$ is spanned by some of regular functions whose multiplicity along $E_0$ is exactly the coefficient of $E_0$ in $-(K_{\bar{X}/X}+E_0-lE)$, that is $lm_0-a_{E_0}(X)$. In particular the value of $Q(l)$ at $l\ge1$ is zero if $lm_0-a_{E_0}(X)<0$. On the other hand $Q(l)$ as a polynomial is of degree $d-1$, whence at least one of $Q(1), \ldots, Q(d)$ is non-zero. Thus $dm_0-a_{E_0}(X)\ge0$, and $\mld_xX \le a_{E_0}(X) \le dm_0 \le d(\sum_im_iE_i\cdot\bar{H}^{d-1})=d\mult_x X$.
\end{proof}

The above argument makes use of the property of Riemann--Roch formula that the intersection number of a divisor with $\bar{H}^{d-1}$, a kind of multiplicity, appears in its top term. It leads us to derive information of the relative canonical divisor $K_{\bar{X}/X}$ from its second-top term, in which the intersection number with $K_{\bar{X}/X}\cdot\bar{H}^{d-2}$ appears. Henceforth $x \in X$ is a germ of a Gorenstein canonical singularity of dimension $d\ge2$. We follow the above setting that $f \colon \bar{X} \to X$ is a log resolution such that the maximal ideal sheaf $\fm_x$ is pulled back to an invertible sheaf $\cO_{\bar{X}}(-E)$. Then $f^*H=\bar{H}+E$ for a general hyperplane section $H$ of $X$ with its strict transform $\bar{H}$. Set $K:=K_{\bar{X}/X}$, $E=\sum_im_iE_i$, $K':=\sum_ia'_iE_i$ and $a'_i+1=a_{E_i}(X)$, where the summations occur over the divisors $E_i$ contracting to the point $x$. The $K'$ is different from $K$ by the divisors which have centres of positive dimension, but $K'\cdot\bar{H}^{d-2}=K\cdot\bar{H}^{d-2}$ as $1$-cycles thanks to the freedom of $\bar{H}$. Let $X_t$ be a scheme obtained from $X$ by cutting out with $d-t$ general hyperplane sections through $x$. Then for $t\ge1$ the intersection $\bar{X_t}$ of their strict transforms on $\bar{X}$ is a log resolution of $X_t$. We discuss on $X_t$ with $t=2,1,0$ essentially because we see the first two terms of Riemann--Roch formula only. Set $S:=X_2$, $C:=X_1$ and $\cO_0:=\cO_{X_0}$. The lemma below is a naive application of Riemann--Roch theorem.

\begin{lemma}\label{lem:naiveRR}
Set $s_l:=\dim f_*\cO_{\bar{S}}(K-lE|_{\bar{S}})/f_*\cO_{\bar{S}}(K-(l+1)E|_{\bar{S}})$. Then
\setlength{\leftmargini}{2em}
\begin{enumerate}
\item\label{itm:EonS}
$(E\cdot\bar{H}^{d-1})=\mult_xX=s_{d-1}-s_{d-2}$.
\item\label{itm:KonS}
$(K'\cdot\bar{H}^{d-1})=(d-1)\mult_xX-2s_{d-2}$.
\end{enumerate}
\end{lemma}

\begin{proof}
$K_{\bar{S}/S}=K-(d-2)E|_{\bar{S}}$.
By the exact sequence
\begin{align*}
0 \to \cO_{\bar{S}}(K_{\bar{S}/S}-(l+1)E|_{\bar{S}}) \to \cO_{\bar{S}}(K_{\bar{S}/S}-lE|_{\bar{S}}) \to \cO_{E|_{\bar{S}}}(K_{E|_{\bar{S}}}-(l+1)E|_{E|_{\bar{S}}}) \to 0
\end{align*}
and the vanishing theorem, $P(l):=\chi(\cO_{E|_{\bar{S}}}(K_{E|_{\bar{S}}}-(l+1)E|_{E|_{\bar{S}}}))$ is equal to $s_{d-2+l}$ for $l\ge0$, whence $P(l)=(s_{d-1}-s_{d-2})l+s_{d-2}$ as a polynomial in $l$. On the other hand Riemann--Roch formula provides $P(l)=-(E|_{\bar{S}})^2l+\frac{1}{2}((K_{\bar{S}/S}-E|_{\bar{S}})\cdot E|_{\bar{S}})$. The lemma follows from comparison of coefficients in $P(l)$.
\end{proof}

This lemma is translated into the language of the curve $C$ or the artinian ring $\cO_0$ by the following \textit{inductive principle}. For a divisor $A$ on $\bar{X_s}$ which is effective outside $f^{-1}(x)$, we set $\cO_{X_t}(A):=f_*\cO_{\bar{X_t}}(A|_{\bar{X_t}}) \cap \cO_{X_t}$ for $1\le t\le s$ (the restriction $\cap \cO_{X_t}$ necessary only when $t=1$) , which consists of regular functions on $X_t$ with multiplicity at least the coefficient in $-A|_{\bar{X_t}}$ along every prime divisor of $\bar{X_t}$. For example $\cO_{X_t}(-E)$ is equal to the maximal ideal sheaf $\fm_x\cO_{X_t}$. Then with a function $h$ in $\fm_x\cO_{X_t}$ defining $X_{t-1}$ we have
\begin{align*}
&\dim\cO_{X_t}(A+E)/\cO_{X_t}(A) \\
=&\dim\cO_{X_t}(A+E)/(h\cO_{X_t}\cap\cO_{X_t}(A+E)+\cO_{X_t}(A)) \\
 &+\dim(h\cO_{X_t}\cap\cO_{X_t}(A+E)+\cO_{X_t}(A))/\cO_{X_t}(A) \\
=&\dim\cO_{X_t}(A+E)\cO_{X_{t-1}}/\cO_{X_t}(A)\cO_{X_{t-1}}+\dim(h\cO_{X_t}\cap\cO_{X_t}(A+E))/(h\cO_{X_t}\cap\cO_{X_t}(A)) \\
=&\dim\cO_{X_t}(A+E)\cO_{X_{t-1}}/\cO_{X_t}(A)\cO_{X_{t-1}}+\dim\cO_{X_t}(A+2E)/\cO_{X_t}(A+E).
\end{align*}
Hence $\dim\cO_{X_t}(A+E)/\cO_{X_t}(A)=\dim\cO_{X_{t-1}}/\cO_{X_t}(A)\cO_{X_{t-1}}$ by the inductive use of it. Set $\cO_0(A):=\cO_C(A)\cO_0$ and $\fm_0:=\fm_x\cO_0$. 

\begin{lemma}\label{lem:RRonC0}
\setlength{\leftmargini}{2em}
\begin{enumerate}
\item\label{itm:EKonC}
$s_l=\dim\cO_C/\cO_C(K-(l+1)E)$ for $l\ge d-2$.
\item\label{itm:Kon0}
$(K'\cdot\bar{H}^{d-1})=(d-1)\dim\cO_0-2\sum_{1\le l\le d-1} \dim\cO_0/\cO_0(K-lE)$.
\end{enumerate}
\end{lemma}

\begin{proof}
By the inductive principle (\ref{itm:Kon0}) follows from (\ref{itm:EKonC}), Lemma \ref{lem:naiveRR}(\ref{itm:KonS}) and $\dim \cO_0=\mult_xX$. For (\ref{itm:EKonC}) it suffices to show that $\cO_S(K-lE)\cO_C=\cO_C(K-lE)$ for $l\ge d-1$, but it is an application of the vanishing theorem to the exact sequence
\begin{align*}
0 \to \cO_{\bar{S}}(K-lE|_{\bar{S}}-\bar{C}) \to \cO_{\bar{S}}(K-lE|_{\bar{S}}) \to \cO_{\bar{C}}(K-lE|_{\bar{C}}) \to 0.
\end{align*}
\end{proof}

We close this section by an important result of the artinian ring $\cO_0$.

\begin{proposition}\label{prp:artinian}
$\cO_0(K-dE)=0$. In particular $\fm_0^d=0$.
\end{proposition}

\begin{proof}
The application of the inductive principle to Lemmata \ref{lem:naiveRR}(\ref{itm:EonS}) and \ref{lem:RRonC0}(\ref{itm:EKonC}) provides $\mult_xX=\dim\cO_C(K-(d-1)E)/\cO_C(K-dE)=\dim\cO_0/\cO_0(K-dE)$.
\end{proof}

\section{Dimension three}
In dimension three we recover the following result without explicit classification, such as \cite{L77} by Laufer, \cite{R75}, \cite{R80}, \cite{R83} by Reid.

\begin{theorem}\label{thm:dim3}
Let $x \in X$ be a Gorenstein terminal three-fold singularity. Then
\setlength{\leftmargini}{2em}
\begin{enumerate}
\item\label{itm:cDV}
A general hyperplane section of $X$ is canonical, proved by Reid in \cite{R83}.
\item\label{itm:BDD3}
$\mld_xX \le 3$, proved by Markushevich in \cite{Mk96}.
\end{enumerate} 
\end{theorem}

The exceptional divisors $E_i$ treated by our numerical argument are only those the restriction of $\bar{H}$ to which is big, which explains the reason of the lack in (\ref{itm:BDD3}) of the characterisation of a smooth three-fold point that it has minimal log discrepancy three. For, a Du Val singularity of type D or E has a minimal resolution on which the strict transform of a general hyperplane section intersects only an exceptional curve along which every non-unit function has multiplicity at least two. Only this curve corresponds to the $E_i$ with $\bar{H}|_{E_i}$ big, which has $a'_i=m_i=2$, in the case when $S$ has such a singularity. Of course one can avoid this difficulty by \cite{EMY03}. Also note that the converse of (\ref{itm:cDV}), for an isolated Gorenstein three-fold singularity, is a simple application of the connectedness lemma.

\begin{proof}
In contrast to the historical context, we prove (\ref{itm:BDD3}) firstly. Lemma \ref{lem:RRonC0}(\ref{itm:Kon0}) with $d=3$ is
\begin{align*}
(K'\cdot\bar{H}^2)=2\dim\cO_0(K-2E)-2\dim\cO_0/\cO_0(K-E),
\end{align*}
which is positive as $X$ is terminal. We use the Gorenstein property of the artinian ring $\cO_0$ that its socle $(0:\fm_0)$ is isomorphic to $k$; a reference is \cite[Theorem 18.1]{Mt89}. The ideal $\cO_0(K-2E)$ is contained in the socle of $\cO_0$ by Proposition \ref{prp:artinian}, whence $\dim \cO_0(K-2E)\le1$. Thus $(K'\cdot\bar{H}^2)$ is positive only if $(K'\cdot\bar{H}^2)=2$, $\cO_0(K-E)=\cO_0$ and $\cO_0(K-2E)\simeq k$. Therefore $\mld_xX-1\le(\sum_ia'_iE_i\cdot\bar{H}^2)=2$, which is (\ref{itm:BDD3}).

We proceed more delicate analysis for (\ref{itm:cDV}). We assume that $X$ is singular, equivalently $\mult_x X\ge2$, since (\ref{itm:cDV}) is trivial when $X$ is smooth. We have already obtained $\cO_C(K-E)=\cO_C$, that is $K|_{\bar{C}} \ge E|_{\bar{C}}$. Also, $\mult_xX=\dim\cO_0=2$ by $\cO_0(K-E)=\cO_0$, $\cO_0(K-2E)\simeq k$ and $\cO_0(K-3E)=0$. Thus $E|_{\bar{C}}$ has the same degree $2$ as $K|_{\bar{C}}$ has, whence $K|_{\bar{C}}=E|_{\bar{C}}$.

We construct the contraction $\bar{S}\to T$ of all curves which have positive coefficients in $K_{\bar{S}/S}$, which is an isomorphism about $\bar{C}$ by $K_{\bar{S}/S}|_{\bar{C}}=K-E|_{\bar{C}}=0$. Write $K_{\bar{S}/S}=P-N$ with effective divisors $P,N$ which have no common components. If $P>0$ there exists an irreducible curve on $\bar{S}$ which has negative intersection number with $P$, hence so with $K_{\bar{S}/S}$; that is a $(-1)$-curve with positive coefficient in $K_{\bar{S}/S}$. By contracting such curves successively, we obtain a smooth surface $T$ with $K_{T/S}\le0$.

We want to prove $K_{T/S}=0$. Suppose not, then there exists an exceptional irreducible curve on $T$ which intersects the support of $K_{T/S}$ properly. It has negative intersection number with $K_{T/S}$, whence it is a $(-1)$-curve. By contracting such curves successively outside a neighbourhood of $\bar{C}$, we finally obtain a smooth surface $T'$ on which there exists an irreducible curve $l_0$ which intersects both $\bar{C}$ and another $l_1$ with negative coefficient in $K_{T'/S}$, where we set $l_i$ as the push-forward of $E_i|_{\bar{S}}$, possibly reducible when $\bar{H}|_{E_i}$ is not big. Then $l_0$ is a $(-1)$-curve, and $(\bar{C}\cdot l_0)_{T'}=(-\sum m_il_i\cdot l_0)_{T'}\le m_0-m_1$. But $m_0\le(E\cdot\bar{H}^2)=2$ and $m_1>a'_1\ge1$, a contradiction.
\end{proof}

\section{Dimension four}
In dimension four we bound minimal log discrepancies in a special case, just for which we introduce one ad hoc definition. For a germ $(x \in X,\Delta)$ of a pair, an algebraic valuation is called a \textit{terminal piece} if it has log discrepancy greater than one and in addition if it is defined by a divisor $F$ on a resolution $\bar{X}$ such that the maximal ideal sheaf is pulled back to an invertible sheaf $\cO_{\bar{X}}(-E)$ and $-E|_F$ is big.

\begin{theorem}\label{thm:dim4}
Let $x \in X$ be a Gorenstein terminal four-fold singularity whose general hyperplane section has a terminal piece. Then $\mult_xX\le2$ and $\mld_xX \le 5-\mult_xX$.
\end{theorem}

We shall use either (PIA) on smooth varieties in \cite{EMY03} or the precise (BDD) for three-folds in \cite{Mk96}, in one place of the proof; otherwise $\mld_xX\le5$ is obtained. First we provide a proposition on all Gorenstein terminal four-fold singularities.

\begin{proposition}\label{prp:dim4}
Let $x \in X$ be a Gorenstein terminal four-fold singularity.
\setlength{\leftmargini}{2em}
\begin{enumerate}
\item\label{itm:mult_atmost2}
If $\mult_xX\le2$ then $\mld_xX \le 5-\mult_xX$.
\item\label{itm:mult_atleast3}
If $\mult_xX\ge3$ then $\mult_xX\ge(K'\cdot\bar{H}^3)$ and $\cO_0(K-4E|_{\bar{C}}+P)\simeq k$ for an arbitrary effective divisor $P$ on $\bar{C}$ with $0<P\le E|_{\bar{C}}$.
\end{enumerate}
In both cases, $0\le(K'\cdot\bar{H}^3)-\dim\cO_0(K-2E)/\cO_0(-2E)\le2$.
\end{proposition}

\begin{proof}
Lemma \ref{lem:RRonC0}(\ref{itm:Kon0}) with $d=4$ is
\begin{align}\label{eqn:KH3}
(K'\cdot\bar{H}^3)
=&\dim\cO_0(K-2E)-\dim\cO_0/\cO_0(K-2E) \\
\nonumber
&+2(\dim\cO_0(K-3E)-\dim\cO_0/\cO_0(K-E)).
\end{align}
The ideal $\cO_0(K-3E)$ is contained in the socle of $\cO_0$ by Proposition \ref{prp:artinian}, whence $\dim\cO_0(K-3E)\le1$. If $\cO_0(K-3E)=0$, then $\dim\cO_0(K-2E)\le1$ by the same reason, and $\cO_0(K-2E)=\cO_0$ by $(K'\cdot\bar{H}^3)>0$, whence $\dim\cO_0=1$. Therefore we assume that $\cO_0(K-3E)\simeq k$. Then by (\ref{eqn:KH3}),
\begin{align*}
\mult_xX-(K'\cdot\bar{H}^3)=2(\dim\cO_0/\cO_0(K-E)+\dim\cO_0/\cO_0(K-2E)-1).
\end{align*}
This is negative if $\mult_xX\le2$ with $\mld_xX\ge4$, and only if $\cO_0(K-2E)=\cO_0$, $\cO_0(-2E)=0$, $\mult_xX\le2$ and $\mld_xX-1\le(K'\cdot\bar{H}^3)\le4$. Then $X$ is a hypersurface singularity, whence (\ref{itm:mult_atmost2}) follows from (PIA) in \cite{EMY03}. Instead, supposing $\mld_xX\ge4$ one can deduce the canonicity of $S$ as in the proof of Theorem \ref{thm:dim3}(\ref{itm:cDV}), then $X_3$ as well as $X$ must be terminal by the connectedness lemma, and actually $X_3$ is smooth by (BDD) in \cite{Mk96} and $\cO_{X_3}(K-2E)\cO_S=\cO_S$ obtained as in the proof of Lemma \ref{lem:RRonC0}(\ref{itm:EKonC}). Anyway we can assume that $\mult_xX\ge(K'\cdot\bar{H}^3)$ besides $\cO_0(K-3E)\simeq k$ henceforth.

We adopt the notation $E\wedge K':=\sum_i\min\{m_i,a'_i\}E_i$ following \cite{BCHM06}. Let $P$ be an arbitrary effective divisor on $\bar{C}$ with $P\le E\wedge K'|_{\bar{C}}$ such that $\cO_0(K-4E|_{\bar{C}}+P)=0$; an example is $P=0$. The ideal $\cO_0(K-3E|_{\bar{C}}+P)$ is contained in the socle of $\cO_0$ but contains $\cO_0(K-3E)\simeq k$, whence $\cO_0(K-3E|_{\bar{C}}+P)\simeq k$. Consider the bilinear form
\begin{align*}
\cO_0(-E)/\cO_0(P-2E|_{\bar{C}}) \times \cO_0(K-2E)/\cO_0(K-3E|_{\bar{C}}+P) \to k,
\end{align*}
which is right non-degenerate because $\cO_0$ has the socle $\cO_0(K-3E|_{\bar{C}}+P)$. In particular $\dim\cO_0/\cO_0(P-2E|_{\bar{C}})\ge\dim\cO_0(K-2E)$. With (\ref{eqn:KH3}) we obtain that
\begin{align}\label{eqn:KH3inequality}
(K'\cdot\bar{H}^3)\le\dim\cO_0(K-2E)/\cO_0(P-2E|_{\bar{C}})+2(1-c_1),
\end{align}
where $c_1:=\dim\cO_0/\cO_0(K-E)\le1$.

We compute $(K'\cdot\bar{H}^3)$ in terms of $P$. Consider the exact sequence
\begin{align*}
0\to\cO_{\bar{C}}(P-3E|_{\bar{C}})\to\cO_{\bar{C}}(K-3E|_{\bar{C}})\to\cO_{K|_{\bar{C}}-P}(K-3E|_{K|_{\bar{C}}-P})\to0.
\end{align*}
We have checked $\cO_C(K-3E|_{\bar{C}})=f_*\cO_{\bar{C}}(K-3E|_{\bar{C}})$ in the proof of Lemma \ref{lem:RRonC0}(\ref{itm:EKonC}). Hence the direct image sheaf $f_*\cO_{\bar{C}}(P-3E|_{\bar{C}})$ also is contained in the structure sheaf $\cO_C$, which means $\cO_C(P-3E|_{\bar{C}})=f_*\cO_{\bar{C}}(P-3E|_{\bar{C}})$. Then the difference $\cO_C(K-3E)/\cO_C(P-3E|_{\bar{C}})$ of them has the same dimension $(K'\cdot\bar{H}^3)-\deg P$ as $\cO_{K|_{\bar{C}}-P}(K-3E|_{K|_{\bar{C}}-P})$ has. Thus with the inductive principle we have
\begin{align*}
&\dim\cO_C(K-2E)/\cO_C(P-2E|_{\bar{C}}) \\
=&\dim \cO_C(K-3E)/\cO_C(P-3E|_{\bar{C}}) \\
&+\dim\cO_C(K-2E)/\cO_C(K-3E)
-\dim\cO_C(P-2E|_{\bar{C}})/\cO_C(P-3E|_{\bar{C}}) \\
=&(K'\cdot\bar{H}^3)-\deg P+\dim\cO_0/\cO_0(K-3E)-\dim\cO_0/\cO_0(P-3E|_{\bar{C}}) \\
=&(K'\cdot\bar{H}^3)-\deg P-(1-c_2),
\end{align*}
where $c_2:=\dim\cO_0(P-3E|_{\bar{C}})\le1$. This principle also computes $\dim\cO_C(K-2E)/\cO_C(P-2E|_{\bar{C}})=\dim\cO_0(K-2E)/\cO_0(P-2E|_{\bar{C}})+c_3$, with $c_3:=\dim\cO_C(K-E)/\cO_C(P-E|_{\bar{C}})\le1$. Therefore
\begin{align}\label{eqn:KH3equality}
(K'\cdot\bar{H}^3)=\dim\cO_0(K-2E)/\cO_0(P-2E|_{\bar{C}})+(1-c_2)+c_3+\deg P,
\end{align}
and the inequalities on $(K'\cdot\bar{H}^3)$ follows from the case $P=0$.

Supposing $\deg P=1$ we shall conclude that $X$ is smooth, which completes the proof. Then $2c_1+c_3\le c_2$ by (\ref{eqn:KH3inequality}) and (\ref{eqn:KH3equality}). In particular $c_1=0$, that is $K|_{\bar{C}}\ge E|_{\bar{C}}$ and in fact $K|_{\bar{C}}=E|_{\bar{C}}$ by the assumption $\mult_xX\ge(K'\cdot\bar{H}^3)$. Hence $\cO_0(P-3E|_{\bar{C}})=\cO_0(K-4E|_{\bar{C}}+P)=0$, that is $c_2=0$. Therefore $c_3=0$, whence $P=E|_{\bar{C}}$ and $\mult_xX=\deg E|_{\bar{C}}=1$.
\end{proof}

\begin{proof}[Proof of Theorem \textup{\ref{thm:dim4}}]
We shall exclude the case (\ref{itm:mult_atleast3}) of Proposition \ref{prp:dim4}. Since $X_3$ has a terminal piece, there exists a divisor $E_0$ with $\bar{H}|_{E_0}$ big and $m_0<a'_0$. Suppose $\mult_xX\ge3$. Then there exists another $E_1$ with $\bar{H}|_{E_1}$ big and $m_1>a'_1$ by $\mult_xX\ge(K'\cdot\bar{H}^3)$ in Proposition \ref{prp:dim4}(\ref{itm:mult_atleast3}). For points $Q_0$ in $E_0|_{\bar{C}}$ and $Q_1$ in $E_1|_{\bar{C}}$, the proposition deduces $\cO_0(K-4E|_{\bar{C}}+m_0Q_0+Q_1)=\cO_0(K-4E|_{\bar{C}}+(m_0-1)Q_0+Q_1)\simeq k$. As the computation of $\dim\cO_C(K-2E)/\cO_C(P-2E|_{\bar{C}})$ in the proof of Proposition \ref{prp:dim4}, we can compute $\dim\cO_C(K-3E|_{\bar{C}}+m_0Q_0+Q_1)/\cO_C(K-3E|_{\bar{C}}+(m_0-1)Q_0+Q_1)=\dim\cO_C(K-4E|_{\bar{C}}+m_0Q_0+Q_1)/\cO_C(K-4E|_{\bar{C}}+(m_0-1)Q_0+Q_1)=1$. Therefore by the same reason as in the proof of Theorem \ref{thm:emb_mult_mld}(\ref{itm:mult_mld}), there exists a regular function $h \in \cO_C(K-3E|_{\bar{C}}+m_0Q_0+Q_1)$ on $C$ whose multiplicity at $Q_0$ is exactly the coefficient of $Q_0$ in $-(K-3E|_{\bar{C}}+m_0Q_0+Q_1)$, that is $2m_0-a'_0$. This happens only if $h$ is a unit in $\cO_C$ since $2m_0-a'_0<m_0$. But $h$ must have multiplicity at $Q_1$ at least the coefficient of $Q_1$ in $-(K-3E|_{\bar{C}}+m_0Q_0+Q_1)$, that is $3m_1-a'_1-1>0$, a contradiction.
\end{proof}

\section{Concluding remarks}
The idea of usage of Riemann--Roch theorem stems from an observation of the simplest case when a Gorenstein singularity $x \in X$ of dimension $d$ has a resolution $f \colon \bar{X} \to X$ with only one exceptional divisor $E_0$, mapped to $x$, such that $-E_0$ is $f$-ample. Then the exact sequence
\begin{align*}
0 \to \cO_{\bar{X}}(K_{\bar{X}/X}-(l+1)E_0) \to \cO_{\bar{X}}(K_{\bar{X}/X}-lE_0) \to \cO_{E_0}(K_{E_0}-(l+1)E_0|_{E_0}) \to 0
\end{align*}
with the vanishing theorem implies that $P(l):=\chi(\cO_{E_0}(K_{E_0}-(l+1)E_0|_{E_0}))$ is equal to $\dim f_*\cO_{\bar{X}}((a_{E_0}(X)-l-1)E_0)/f_*\cO_{\bar{X}}((a_{E_0}(X)-l-2)E_0)$ for $l\ge0$, whence $a_{E_0}(X)\le d$ as $P(l)$ is a polynomial of degree $d-1$ in $l$. This argument looks similar to that of Theorem \ref{thm:emb_mult_mld}, but in fact is completely different by the reason of the choice of $E_0$, a divisor divided by multiplicity, in place of $E=m_0E_0$. As far as our approach treats divisors such as $E$ appearing in the pull-back of ideal sheaves, it will derive properties of \textit{log canonical thresholds}, or more generally \textit{jumping coefficients} in \cite{ELSV04}, rather than those of the minimal log discrepancy, because it analyses the values $l$ in divisors of form $K-lD$ with $K$ canonical divisor, which encode information of $K$ divided by $D$, corresponding to thresholds. For example it implies that log canonical thresholds are bounded by dimension. This philosophy is reflected also by the aspect in positive characteristic that log canonical threshold has its correspondence, \textit{$F$-pure threshold} in \cite{TW04}, in contrast to minimal log discrepancy.

Theorems \ref{thm:dim3} and \ref{thm:dim4} are in virtue of the nature of the singularities concerned that they are characterised in terms of surfaces obtained by cutting out with general hyperplane sections, on which one can handle full Riemann--Roch formula without higher chern classes. Therefore there seem to exist two directions to advance our approach. One is to face also the lower terms in this formula, which represents analysis of cut-out varieties of higher dimension, whereas the other is to treat also divisors from non-maximal ideal sheaves, which represents analysis of special hyperplane sections.

\bibliographystyle{amsplain}

\end{document}